\documentclass{article}

\usepackage[top=2.54cm, bottom=2.54cm, outer=2.54cm, inner=2.54cm, heightrounded, marginparwidth=2cm, marginparsep=0.25cm]{geometry}

\usepackage{amsmath}
\usepackage{amsfonts}
\usepackage{amssymb}
\usepackage{amsthm}
\usepackage{times}
\usepackage{extarrows}
\usepackage{mathrsfs}
\usepackage{caption}
\usepackage{indentfirst}
\usepackage{fullpage}
\usepackage{mathtools}
\usepackage{xcolor}

\usepackage{palatino}

\usepackage{ucs}
\usepackage[utf8x]{inputenc}
\usepackage{graphicx}
\usepackage{epstopdf}
\usepackage{tocloft}
\newcommand{\diam}{\textrm{diam }}

\DeclareMathOperator{\Var}{Var}

\renewcommand{\S}{\mathcal{S}}

\makeatletter
\newtheorem*{rep@theorem}{\rep@title}
\newcommand{\newreptheorem}[2]{%
\newenvironment{rep#1}[1]{%
 \def\rep@title{#2 \ref{##1}}%
 \begin{rep@theorem}}%
 {\end{rep@theorem}}}
\makeatother

\newtheorem{theorem}[equation]{Theorem}
\newtheorem{lemma}[equation]{Lemma}

\newtheorem{remark}[equation]{Remark}
\newtheorem{definition}[equation]{Definition}

\newtheorem{Definition}[equation]{Definition}

\newtheorem{Proposition}[equation]{Proposition}

\newtheorem{Corollary}[equation]{Corollary}
\newreptheorem{theorem}{Theorem}
\theoremstyle{definition}

\newtheorem{Remark}[equation]{Remark}

\title{Shuffling large decks of cards and the Bernoulli-Laplace urn model}

\author{
  Nestoridi, Evita
  \and
  White, Graham
}

\date{\today}

\begin{document}

\maketitle

\tableofcontents

\section{Introduction}
\label{sec:intro}

People are often required to shuffle cards by hand. In many cases, this involves a relatively small deck of between for example 20 and 78 cards. Sometimes, however, it is required to randomise a larger deck of cards. This occurs in casinos, where games such as blackjack are played using eight standard decks of cards to reduce the effectiveness of card counting. It also occurs in some designer card and board games, such as Ticket to Ride or Race for the Galaxy. In such games, a deck of between 100 and 250 cards may be used. It is easy to randomise small decks of cards via riffle shuffles (\cite{Bayer}). Riffle shuffles would still be effective for larger decks of cards, but may be more difficult to implement, because they require manipulating two halves of the deck, each held in one hand. We will discuss schemes for shuffling a large deck of cards by manipulating fewer cards at a time. Similar considerations may also occur in cryptography problems, such as those discussed in \cite{Defense}.

The standard casino method for shuffling a deck of $N$ cards is to cut it into $p$ piles, shuffle each pile, recombine them, and then do a deterministic step which moves cards between the sections of the deck corresponding to the various piles, such as cutting $k$ cards from the top of the deck to the bottom. This entire process is repeated several times. This paper gives good bounds on the rates of convergence for these schemes, under the assumption that at each step, the individual piles are perfectly shuffled. At the same time, it turns out that the necessary analysis also solves a classical problem, the Bernoulli-Laplace model with $p$ urns, for even $p$.

\subsection{The Bernoulli-Laplace urn problem and its early history}
Daniel Bernoulli (1770) studied many different urn problems, including some that we consider. For instance, he considered a three-urn problem where the first urn contains $n$ white tickets, the second $n$ black tickets and the third $n$ red tickets. A ticket is drawn at random from the first urn, and then placed into the second. Then a ticket is picked from the second urn and placed into the third. Finally a ticket is drawn from the third urn and deposited in the first. Bernoulli found the expression for the expected number of white tickets in each urn after this process is repeated multiple times (\cite{D.Bernoulli_three}). A summary of Bernoulli's work can be found in Todhunter \cite{Todhunter} (page 231). An English translation of the Latin can be found at \cite{Bernoulli_site}, along with commentary \cite{Pulskamp}.

The following problem was introduced by Bernoulli, but studied in depth by Laplace (see page $378$ of \cite{Feller}). Consider two urns, each containing $n$ balls, with $n$ of the balls colored white and $n$ red. Draw one ball at random from each urn, and exchange them. Laplace was interested in finding the distribution of the number of white balls in the first urn after $r$ steps and solved this problem, introducing ``the first differential equation of the type now known from the probabilistic theory of diffusion processes'', as explained by Jacobson in \cite{MJ}.  Some authors objected to Laplace's methods and level of rigour --- for more details see \cite{MJ}. 

Later, Markov \cite{Markov} constructed the first Markov chain using a more general model. He considered the Markov chain where one draws a ball from each of two urns, the first having $a$ balls, the second having $b$ balls among which $p(a+b)$ are white and $q(a+b)$ are red, with $0<p<1$ and $q=1-p$, and found that the stationary distribution is hypergeometric. For $a=b$ and $p=q=\frac{1}{2}$ this is the classical Bernoulli-Laplace problem. In this paper, we are concerned with the mixing time of a generalisation of the Bernoulli-Laplace problem. We consider $2p$ urns, each containing $n$ balls. At each step, $k$ balls are chosen at random from each urn and moved to the next urn. Initially, each urn contains $n$ balls of the same colour, a different colour for each urn.

\subsection{A card shuffling model}

Our results are easiest to describe in the case of two piles. A more careful description of the procedure follows.

\begin{Definition}
\label{def:twopile}
A straightforward approach to shuffling a deck of $2n$ cards is to cut the deck into two piles, perfectly shuffle each independently and then stack the two piles atop one another. We then move a fixed number $k$ of cards from the top of the deck to the bottom and repeat this entire procedure. We will refer to this shuffling scheme as the \emph{two pile, $k$--cut shuffle}. See Figure \ref{fig:twopile} for an illustration of this procedure.
\end{Definition}

\begin{figure}[ht!]
\centering
\includegraphics[scale=0.4]{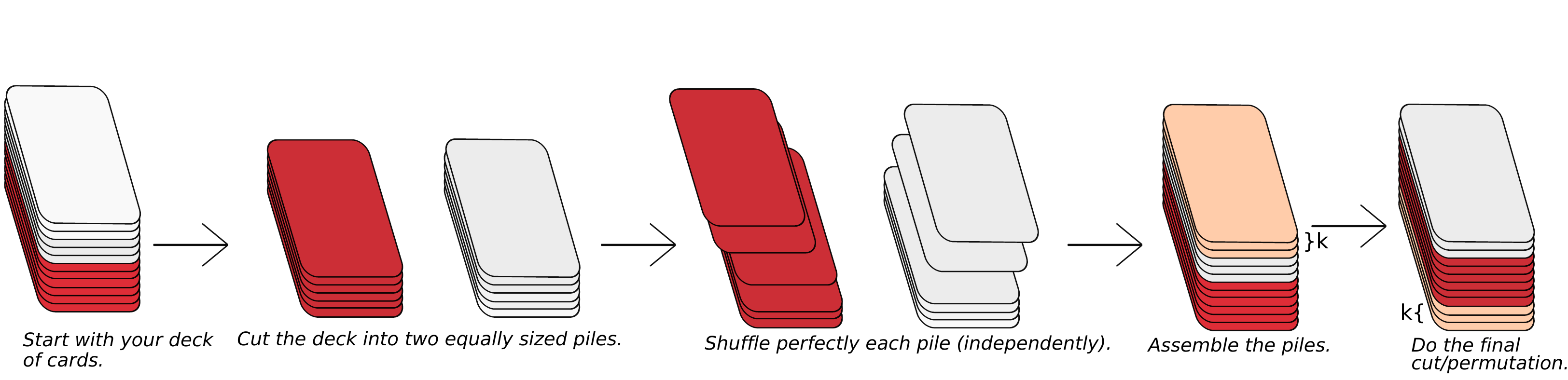}
\caption{A two pile, $k$--cut shuffle.}
\label{fig:twopile}
\end{figure}

One might expect that if this shuffle is repeatedly applied to a deck of cards, then the deck will eventually become well shuffled. That is, that it would be very close to equally likely to be in any of the $(2n)!$ possible configurations. We will analyse this algorithm, and determine how many iterations are required for the deck to become close to `random'.

Our scheme defines a random walk on the symmetric group $S_{2n}$. However, the first step is to perfectly shuffle the top $n$ cards of the deck. Henceforth, any of the $n!$ arrangements of those cards will be equally likely. Similarly, the other $n$ cards are perfectly shuffled. Therefore, we need not track the values of the cards, just which were originally in the top $n$ positions and which in the bottom $n$. To encode this information, colour the top $n$ cards black and the bottom $n$ cards red. 

The final step, of moving $k$ cards from the top of the deck to the bottom, is the same as taking $k$ cards from the top half of the deck, and putting them in the bottom half, and simultaneously choosing $k$ cards from the bottom half of the deck, and moving them to the top half. It doesn't matter which positions these cards were in within the halves, because each half was perfectly shuffled, and it doesn't matter which positions the cards are placed in, because each half is about to be perfectly shuffled again. All that matters is which card is in which half of the deck. 

We have rephrased our scenario as a Bernoulli-Laplace model. That is, we have two urns (halves of the deck), each containing $n$ balls, some red and some black. At each step, $k$ balls are chosen from each urn, and simultaneously moved to the other urn. This model is illustrated in Figure \ref{fig:BL}.

\begin{figure}[ht!]
\centering
\includegraphics[scale=0.7]{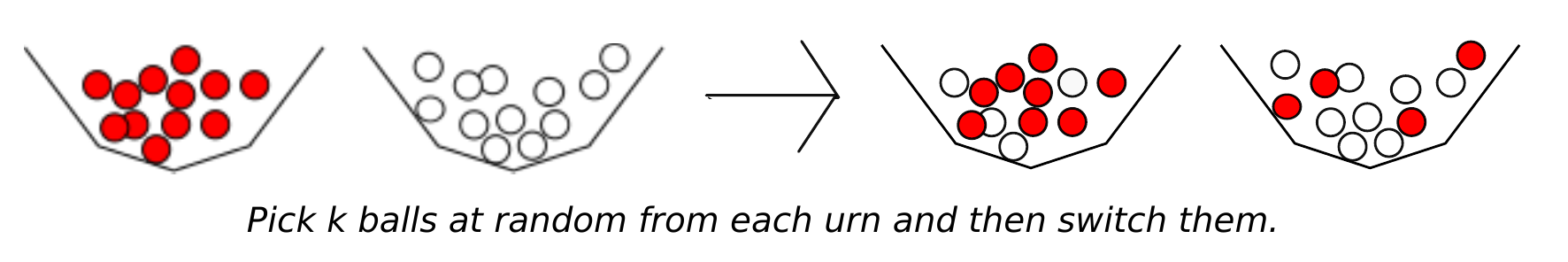}
\caption{The two pile, $k$--cut shuffle, reinterpreted as a Bernoulli-Laplace model}
\label{fig:BL}
\end{figure}

In total, there are $n$ red balls and $n$ black balls, so the state of the system is completely determined by the number of red balls in the right urn. We shall denote by $i$ the state of having $i$ red balls in the right urn. The transition matrix and stationary distribution of this chain are given by

\begin{lemma}
This Markov chain is reversible with transition matrix 
$$K(i,j)= \begin{cases} \sum^{ \min \lbrace k-(j-i),n-j,i \rbrace}_{m=0} \frac{{i \choose m} {n-i \choose k-m} {n-i \choose j-i+m} {i \choose k-(j-i) -m} }{{n \choose k}^2} & \mbox{if } i \leq j \\ \sum^{\min \lbrace k,n-j,i \rbrace}_{m=i-j} \frac{{i \choose m} {n-i \choose k-m} {n-i \choose m-(i-j)} {i \choose k+(i-j)-m} }{{n \choose k}^2} & \mbox{if } i>j \end{cases}$$
and stationary distribution the hypergeometric distribution: 
$$\\ \pi_n(j)= \frac{{n \choose j} {n \choose n-j}}{ {2n \choose n}}\quad 0 \leq j \leq n.$$
\end{lemma}

Let $P^{*t}$ be the probability distribution of this process, starting at the configuration with $n$ red balls in the left urn
and $n$ black balls in the right urn, after $t$ steps. Distance to stationarity will be measured by the total variation distance 
$$||P^{*t} - \pi_n||_{T.V.}= \frac{1}{2} \sum_j  |P^{*t}(j) - \pi_n(j)|= \max_{A \subset X } \{ P^{*t} (A)- \pi_n(A) \}$$

where $X=\{ 0,1,2,\dots,n \}$ is the state space.

The mixing time of the walk, $t_{\text{mix}}(\epsilon)$ is defined to be the maximum over starting positions of:
$$\min \lbrace t \in \mathbb{N} :  ||P^{*t} - \pi_n|| < \epsilon \rbrace.$$

The two pile $k$--cut shuffle describes a random walk on the symmetric group. To describe the convergence of this shuffling scheme, we will obtain bounds on the mixing time of this walk.
\subsection{Results}

Diaconis and Shashahani proved in \cite{Bernoulli-Laplace} that if $k=1$ then the mixing time is $\frac{1}{4}n\log n+cn$, and that the walk has a cut-off at that time. The results obtained in the present paper generalise this result. In particular, a coupling argument  gives the following result: 

\begin{theorem}\label{general two piles upper bound}
For the two pile, $k$--cut shuffle
\begin{itemize}
\item[(a)]If $\lim_{n \longrightarrow \infty}\frac{k}{n} = 0$, then asymptotically $$t_{\text{mix}}(\epsilon) \leq \frac{n}{2k}\log \left(\frac{n}{\epsilon} \right)$$
\item[(b)] If $\lim_{n \longrightarrow \infty}\frac{k}{n} = b \leq \frac{1}{2}$, then asymptotically $$t_{\text{mix}}(\epsilon) \leq \frac{1}{2b(1-b)}\log \left(\frac{n}{\epsilon} \right)$$
\end{itemize}
\end{theorem}

We derive a similar lower bound, as follows:
\begin{theorem}\label{lower two piles}
If $\lim_{n \rightarrow \infty} \frac{k}{n}= d \in [0,\frac{1}{2})$, then for $0<c< \frac{1}{4} \log (n)$ and $t = \frac{n}{4k} \log (n) - \frac{cn}{k}$ there exists a universal constant $D$ such that 
$$||P^{*t} - \pi_n||_{T.V.} \geq 1- \frac{D}{e^{4c}}.$$
\end{theorem}

This paper contains bounds for the mixing time in the case where $k$ is close to $\frac{n}{2}$, using spherical function theory:
\begin{theorem}\label{case n/2}

 If $n$ is even and $k=\frac{n}{2} - c$, where $0< c \leq \log_6 n$  then
$$4||P_0^{*t} - \pi_n||^2_{T.V.} \leq B \pi^2 n^2 \left( \frac{A}{ n}\right)^{2t-2}$$
where $A=\begin{cases}1 & \mbox{ if }k=\frac{n}{2} \mbox{ or } \frac{n}{2}-1\\
6^{c-1} & \mbox{ otherwise}  \end{cases}$
and 
$B=\begin{cases}1/6 & \mbox{ if }k=\frac{n}{2} \\
1/3 & \mbox{ if }k=\frac{n}{2} -1\\
12 & \mbox{ otherwise}  \end{cases}$

\end{theorem}

\begin{table}[!ht]
\centering
\small
\tabcolsep=0.11cm
\begin{tabular}{c c c c }
\hline\hline
number of balls & $\substack{ \mbox{number of balls} \\  \mbox{picked from each urn }}$ & mixing time & T.V. distance after $t$ \\
\hline
$2n$ & $k$  &$ t$ & $ ||P_0^{*t} - \pi_n||_{T.V.} \leq  $   \\ [0.5ex] 
\hline
104 & 26 & 3 & $6.1 \times 10^{-4}$ \\
104 & 2 & $ 13 \log 520 \approx 81$ & $\frac{1}{10}$ \\
416 & 104 & 3 &  $3.8 \times 10^{-5}$\\
416 & 6 &  $ 17.3 \log 2080 \approx 132$  &   $\frac{1}{10}$\\
2080 & 520 &  3  &   $1.5 \times 10^{-6}$\\
2080 & 45 &  $ 11.5 \log 10400 \approx 107$  &   $\frac{1}{10}$\\ [1ex]
\hline
\end{tabular}
\caption{Examples of the mixing time for different parameter values. In particular, it is important whether $k$ is near $\frac{n}{2}$ and how large the ratio $\frac{n}{k}$ is.}
\end{table}

We prove Theorem \ref{case n/2} in Section \ref{sec:closeupper} using an explicit diagonalisation via dual Hahn polynomials, while in Section \ref{sec:coupling2} we prove Theorem \ref{general two piles upper bound} using path coupling techniques. Finally, we prove Theorem \ref{lower two piles} in Section \ref{sec:lower} using second moment methods.

We will also consider the following natural generalisation of the two pile, $k$--cut shuffle.

\begin{Definition}
\label{def:apile}
To perform a \emph{$2p$--pile, $k$--cut shuffle}, we split the deck into $2p$ equal piles, perfectly shuffle each of them, put the piles back atop each other in their original order, move $k$ cards from the top of the deck to the bottom, and repeat this process.
\end{Definition}

In the same way as the 2--pile, $k$--cut shuffle can be seen as an urn model where at each step $k$ balls are moved from each of two urns to the other, a $2p$--pile, $k$--cut shuffle can be realised as an urn model with $2p$ urns arranged in a circle, where at each step $k$ balls are moved from each urn to the next.

As with the previous shuffling scheme, we will be interested in the mixing time of this random walk as a function of $p$ and $k$. In Sections \ref{sec:urnscycle} and \ref{sec:genlower}, we prove the following

\begin{theorem}
\label{thm:morepiles}
If we have $2p$ urns, each containing $n$ balls, and simultaneously move $k$ balls from each urn to the next, then for $c>0$,
\begin{itemize}
\item[(a)] $t_{\text{mix}}\left(\frac{e^{-c}}{4}\right) \approx \frac{8np^2}{\pi^2k}\log (16np)+ c  \frac{8np^2}{\pi^2k }$
\item[(b)] $t_{\text{mix}}(e^{-c}) \geq c \frac{2p^2n}{ \pi^2k} .$
\end{itemize}
\end{theorem}

\begin{table}[!ht]
\centering
\begin{tabular}{c c c c }
\hline\hline
total number of balls & number of urns & $\substack{\mbox{number of balls} \\ \mbox{picked from each urn}}$  & mixing time \\
\hline
$2np$ & $2p$ & $k$  & $t_{\text{mix}} \leq$ \\ [0.5ex] 
\hline
104 & 4 & 13 & 11  \\
104 & 4 & 2 &  71 \\
4160 & 10 & 102 & 220 \\
4160 & 10 & 6 & 3700 \\ [1ex]
\hline
\end{tabular}
\caption{Some examples of the upper bound of Theorem \ref{thm:morepiles} for different values of the parameters.}
\end{table}
\begin{remark}
Theorem \ref{thm:morepiles} applies to the case where the deck is cut into an even number of piles. A similar upper bound can be obtained for the case of an odd number of piles. If the deck is cut into $2p+1$ of piles of $n$ cards, then  add $n$ marked cards to the top of the deck, perform the shuffle with $2p+2$ piles, and then remove the marked cards. 
\end{remark}

\section{Generalisations with multiple piles}
\label{sec:genpiles}

\subsection{Coupling for a random walk on the cycle}
\label{sec:cycle} 

In this section, we consider a simple random walk on the cycle. Problems similar to this are considered in, for example, Example 2.1 of \cite{Threads}. Our goal in this section is to introduce techniques which will allow us in Section \ref{sec:urnscycle} to study the $k$--cut $2p$--pile shuffle.

Consider a particle moving on the vertices of a cycle of length $2y$. With probability $\frac{1}{4}$ it jumps one position to the left or the right, otherwise it remains in place (with probability $\frac{1}{2}$). We will use coupling to find an upper bound on the mixing time for this random walk.

Each state is specified by the position of the particle, which is a number modulo $2y$. Consider the metric 
\begin{equation}\label{eq:defd}d(i,j) = \frac{\sin(|i-j|\frac{\pi}{2y})}{\sin(\frac{\pi}{2y})}.\end{equation}

We shall not use that $d$ satisfies the triangle inequality, but this follows from $d$ being a concave function applied to the usual graph metric on the cycle.

\begin{Remark}
\label{rem:eitherarc}
The function $d(i,j)$ depends only on $|i-j|$, and gives the same result if $2y-|i-j|$ is used instead, so it doesn't matter which way around the cycle distances are measured. Nonetheless, it's easier to consider all distances to be measured along the shorter arc. We will use $|i-j|$ to refer to this shorter distance.
\end{Remark}

The constant on the denominator  of Equation \ref{eq:defd} scales the function $d$ so that $d(i,i+1) = 1.$ 

In order to apply coupling, we consider two instances of the random walk, in states $X$ and $Y$ respectively. We couple the two walks so that when they are not in the same state, if one of them moves, the other does not. 
If they are in the same state, then they move in the same way. This coupling appears as Example 5.3.1 of \cite{LPW}. As the walks evolve according to this coupling, we examine the expected value of the function $d(X,Y)$. Our goal is to show that after a certain time, there is a high probability that the two chains are in the same state. It suffices to work with the starting configuration where the two chains are in exactly opposite states, because the evolution of this joint state must pass through states equivalent to every other joint state before being in a state where the two chains match. Hence an upper bound for this initial configuration will also bound the convergence of any other initial configuration.

\begin{Remark}
\label{rem:initE}
At time $t = 0$, the maximum possible value of $E(d(X,Y))$ is $$\frac{\sin(\frac{n\pi}{2y})}{\sin(\frac{\pi}{2y})} = \frac{1}{\sin(\frac{\pi}{2y})} \approx \frac{2y}{\pi}$$
\end{Remark}

\begin{Proposition}
\label{prop:decayE}
If $X_k$ and $Y_k$ are the states of the random walks $X$ and $Y$ after $k$ steps, then we have that $E(d(X_{k+1},Y_{k+1})) = \cos(\frac{\pi}{2y})E(d(X_k,Y_k))$
\end{Proposition}
\begin{proof}
We will work with each possible pair of definite states $(X_k,Y_k)$. We need only consider the cases where $d(X_k,Y_k) \neq 0$, because the other cases do not contribute to $E(d(X_k,Y_k))$ or to $E(d(X_{k+l},Y_{k+l}))$ for any $l \geq 0$. If $d(X_k,Y_k)$ is nonzero, then there is a $\frac{1}{2}$ chance that $X$ and $Y$ move away from one another, and a $\frac{1}{2}$ chance that they move towards one another, unless they are at opposite vertices, in which case they must move towards one another.

We now calculate $E(d(X_{k+1},Y_{k+1}))$. Let $\delta = |X_k - Y_k|$.

If $0 < \delta < y$, then $|X_{k+1} - Y_{k+1}|$ is equally likely to be $\delta - 1$ and $\delta + 1$. Therefore, we calculate 
\begin{align*}
E(d(X_{k+1},Y_{k+1})) &= \frac{1}{2}\frac{\sin((\delta+1)\frac{\pi}{2y})}{\sin(\frac{\pi}{2y})}+\frac{1}{2}\frac{\sin((\delta-1)\frac{\pi}{2y})}{\sin(\frac{\pi}{2y})}\\
%&= \frac{\sin((\de+1)\frac{\pi}{2y})+\sin((\de-1)\frac{\pi}{2y})}{2\sin(\frac{\pi}{2y})} \\
%&= \frac{2\sin(\de\frac{\pi}{2y})\cos(\frac{\pi}{2y})}{2\sin(\frac{\pi}{2y})} \\
%&= \cos\left(\frac{\pi}{2y}\right)\frac{\sin(\de\frac{\pi}{2y})}{\sin(\frac{\pi}{2y})} \\
&= \cos\left(\frac{\pi}{2y}\right)d(X_k,Y_k)
\end{align*}

If $\delta = y$, then by Remark \ref{rem:eitherarc}, a similar calculation shows that in this case too, \begin{equation}\label{eq:scalefactor}E(d(X_{k+1},Y_{k+1})) = \cos\left(\frac{\pi}{2y}\right)d(X_k,Y_k)\end{equation}
We have shown that for any definite pair of states $(X_k,Y_k)$, Equation \ref{eq:scalefactor} is true. Combining these results across any mixture of states $(X_k,Y_k)$, we get that $$E(d(X_{k+1},Y_{k+1})) = \cos(\frac{\pi}{2y})E(d(X_k,Y_k))$$
This completes the proof.
\end{proof}

\begin{Corollary}
\label{cor:couple1}
We have that for $t = cy^2\frac{4}{\pi^2}$, with $c>0$ $$E(d(X_t,Y_t)) \leq e^{-c}\frac{y}{2\pi}.$$
\end{Corollary}
\begin{proof}
Combining Remark \ref{rem:initE} and Proposition \ref{prop:decayE}, we have that to first order, $$E(d(X_t,Y_t)) \leq \cos^t\left(\frac{\pi}{2y}\right)\cdot\frac{2y}{\pi}.$$

Using that $\cos(x) \approx 1 - x^2$ for small $x$ and that $(1 - \frac{1}{k})^k \approx e^{-1}$ for large $k$, we have that for $t \approx  cy^2\frac{4}{\pi^2},$

\begin{align*}
E(d(X_t,Y_t)) &\leq \cos^t\left(\frac{\pi}{2y}\right)\cdot\frac{2y}{\pi}\\
%&\approx \left(1-\frac{\pi^2}{4y^2}\right)^t \cdot \frac{2y}{\pi}\\
&\approx \left(1-\frac{\pi^2}{4y^2}\right)^{\frac{c4y^2}{\pi^2}} \cdot \frac{2y}{\pi}\\
&\approx e^{-c}\frac{2y}{\pi}
\end{align*}

This completes the proof.
\end{proof}

At this point, we will use our earlier assumption that the initial configuration has the two chains in exactly opposite states.

\begin{Proposition}
\label{prop:exn2}
At any time $t$, the expected value of $|X_t-Y_t|$, conditioned on $X_t \neq Y_t$ is at least $\frac{y}{2}$.
\end{Proposition}
\begin{proof}
At each time $t$, consider the function $f_t(k) = P(X_t-Y_t = k \neq 0)$. These functions are alternately supported on the sets of even and of odd integers. For this proof, we will measure the quantity $X_t - Y_t$ in a consistent direction, such as clockwise, rather than always using the shortest arc. Hence this quantity can take values in ${0, 1, \dots, 2y-1,2y}$, with $2y$ treated the same as zero. The distance $|X_t-Y_t|$ will still be measured along the shorter arc.

We will show that each $f_t$ is concave on its support. This is true of the function $f_0$, which is zero everywhere except $y$, where it is one, and of the function $f_1$, which is $\frac{1}{2}$ at $y-1$ and $y+1$, and zero elsewhere.

Each function $f_t$ is obtained from a previous one by $f_t(k) = \frac{1}{4}f_{t-2}(k-2)+\frac{1}{2}f_{t-2}(k)+\frac{1}{4}f_{t-2}(k+2)$. This recurrence preserves the property that $f_t$ is concave on its support, so each $f_t$ has this property. Each function $f_t$ is symmetric about $x=y$, so has a maximum there if $t$ is even and maxima at $x = y \pm 1$ if $t$ is odd. Hence, the distribution of $|X_t-Y_t|$ according to $f_t$ dominates the distribution of $|X_t-Y_t|$ uniform on the support of $f_t$. 

Therefore the expected value of $|X_t-Y_t|$, conditioned on $X_t \neq Y_t$, is at least what it would be if $f_t$ were constant on its support. The calculations are slightly different depending on the parities of $y$ and $t$. For example, if $y$ and $t$ are both even:
\begin{align*}
E(|X_t-Y_t|:X_t \neq Y_t)) & \geq \frac{1}{y-1}(2 + 4 + \dots + (y-2) + n + (y-2) + \dots + 4 + 2) \\
& = \frac{y^2}{2(y-1)} \\
& > \frac{y}{2}
\end{align*}

Similar calculations give the result in the three other cases, completing the proof.
\end{proof}

\begin{Corollary}
\label{cor:exn2}
At any time $t$, the expected value of $d(X_t,Y_t)$, conditioned on $X_t \neq Y_t$, is at least $\frac{y}{\pi}$.
\end{Corollary}
\begin{proof}
From the definition of $d$ in Equation \ref{eq:defd}, we calculate that $d(X_t,Y_t) \geq \frac{2}{\pi}|X_t - Y_t|$.
\end{proof}

\begin{Corollary}
We have that for $c>0$ $$t_{\text{mix}}\left(2e^{-c} \right) \lesssim cy^2\frac{4}{\pi^2}.$$ 
\end{Corollary}
\begin{proof}
From Corollary \ref{cor:couple1} and Corollary \ref{cor:exn2}, we have that $$P(X_t \neq Y_t) \leq \frac{E(d(X_t,Y_t))}{E(d(X_t,Y_t) \vert X_t \neq Y_t)} \leq 2e^{-c}.$$ Theorem 5.2 of \cite{LPW} gives the result.
\end{proof}

\begin{Remark}
\label{rem:58}
In the proof of Proposition \ref{prop:exn2}, a more careful calculation of the limiting behaviour of the $f_t$, which approach a parabolic distribution proportional to $p(k) = k(2y-k)$ would find that the expected value of $|X_t-Y_t|$ is at least $\frac{5y}{8}$. In the proof of Corollary \ref{cor:exn2}, we could bound the ratio $\frac{d(X_t,Y_t)}{|X_t - Y_t|}$ more carefully and improve the result slightly. However either of these improvements would lead only to decreasing the constant factor inside the logarithm.
\end{Remark}

\subsection{Coupling for urns arranged in a cycle}
\label{sec:urnscycle}
We now consider $2p$ urns arranged in a cycle, each urn containing $n$ balls. At each step, $k$ balls at random are selected from each urn and moved to the next urn. We will use a coupling argument based on the techniques of Section \ref{sec:cycle} to find an upper bound for the mixing time of this random walk.

Consider two copies of the system, and label the balls in each with the numbers $1$ to $2np$. We now describe our coupling. 

To begin, we will pair each ball in the first chain with a ball in the same urn in the second chain. For each ball in the first chain, if the ball labelled by that number is in the same urn in the second chain, pair them. For each urn, there are now an equal number of unpaired balls in the first and second chain --- pair them arbitrarily. 

Next, we describe the coupling for one step of the joint chain. In the first chain, choose $k$ balls from each urn, and move each of these balls to the next urn, as usual. For each ball chosen this way, in the second chain, move its partner to the next urn. Restricted to either chain, this step is the usual step of choosing $k$ balls from each urn, so the coupling is valid.

We now forget the pairing, and construct a new pairing according to the same rules, because there now may be more matching pairs of balls, take a step using this new pairing, and repeat. Note that this coupling has the property that whenever two balls have the same number and are in the same urn, they move together or not at all, so matches are never destroyed. 

We now analyse the time required for this coupling to bring the two copies of the chain together, using the techniques of Section \ref{sec:cycle}.

For each $i$, if the two balls labelled by $i$ are in the urns numbered $x_i$ and $y_i$, then as in Equation \ref{eq:defd}, let $$d_i = \frac{\sin(|x_i-y_i|\frac{\pi}{2p})}{\sin(\frac{\pi}{2p})}.$$

Let $D$ be the sum of all $2np$ variables $d_i$. We will perform similar calculations to those of Section \ref{sec:cycle} using $D$. All calculations are to first order in $p$, using the same approximations as in that section.

\begin{Proposition}
The maximum initial value of $D$ is $\frac{4np^2}{\pi}$.
\end{Proposition}
\begin{proof}
The quantity $D$ is the sum of $2np$ variables, each of which is at most $\frac{2p}{\pi}$, as in Remark \ref{rem:initE}.
\end{proof}

\begin{Proposition}
\label{prop:exd}
The expected value of $D$ after $t$ steps is $$E(D_t) = \left(\frac{k^2 + (n-k)^2}{n^2} + \frac{2k(n-k)}{n^2}\cos \left(\frac{\pi}{2p} \right)\right)^tD_0.$$
\end{Proposition}
\begin{proof}
The quantity $D$ is the sum of $2np$ variables $d_i$. After a single step, each $d_i$ is unchanged if both or neither of the balls with that label are moved, which happens with probability $\frac{k^2 + (n-k)^2}{n^2}$, and decreases in expectation by a factor of $\cos(\frac{\pi}{2p})$ if exactly one of the balls with that label moves, as in Proposition \ref{prop:decayE}.
\end{proof}

\begin{Corollary}
\label{cor:couple2}
After time $$t \approx \frac{4n^2p^2}{\pi^2k(n-k)}\log (16np)+ c  \frac{4n^2p^2}{\pi^2k(n-k)},$$where $c>0$, the expected value of $D$ is at most $\frac{e^{-c}p}{4\pi}$. 
\end{Corollary}
\begin{proof}
The calculation is the same as that in Corollary \ref{cor:couple1}.
\end{proof}

\begin{reptheorem}{thm:morepiles}(a)
Let $c>0$ then $t_{mix}\left(\frac{e^{-c}}{4}\right) \lesssim \frac{8np^2}{\pi^2k}\log (16np)+ c  \frac{8np^2}{\pi^2k }.$
\end{reptheorem}
\begin{proof}
Noting that Corollary \ref{cor:exn2} applies to each of the $d_i$, this follows from Corollary \ref{cor:couple2} that $$t_{mix}\left(\frac{e^{-c}}{4}\right) \lesssim \frac{4n^2p^2}{\pi^2k(n-k)}\log (16np)+ c  \frac{4n^2p^2}{\pi^2k(n-k)},$$

Using the fact that $k \leq \frac{n}{2}$, this implies the simpler bound 
$$t_{mix}\left(\frac{e^{-c}}{4}\right) \lesssim \frac{8np^2}{\pi^2k}\log (16np)+ c  \frac{8np^2}{\pi^2k }.$$
\end{proof}
As in Remark \ref{rem:58}, it is possible to improve the constant factor inside the logarithm.

The methods of this section also apply when the number of balls moved from each urn is not a constant, but rather is chosen from some probability distribution. For example, if we fix $a \leq \frac{n}{4}$ and then at each step either move $a$ balls from each urn or $2a$ balls from each urn, each possibility with probability $\frac{1}{2}$, then in Proposition \ref{prop:exd}, the expected value of $D$ would at each step be multiplied by the average of the factors $$\left(\frac{k^2 + (n-k)^2}{n^2} + \frac{2k(n-k)}{n^2}\cos \left(\frac{\pi}{2p} \right)\right)$$ with $k=a$ and with $k = 2a$.

\subsection{Lower bounds}
\label{sec:genlower}

In this section, we will consider the same scenario as in Section \ref{sec:urnscycle} --- that is, $2p$ urns arranged in a cycle, each urn containing $n$ balls of the same color. At each step, $k$ balls at random are selected from each urn and moved to the next urn. We will give a lower bound for the mixing time which differs from the upper bound by a factor of $\log(np)$.

\begin{reptheorem}{thm:morepiles}(b)
Let $t= c \frac{2p^2n}{ \pi^2k},$ where $c>0$. Then 
$$ ||P_{x_0}^{*t} - \pi_n||_{T.V.}\leq e^{-c}$$
for every starting configuration $x_0$.
\end{reptheorem}

\begin{proof}

To prove this lower bound, it will be easier to consider the card shuffling model. That is, to take a single step of the chain, a deck of $2pn$ cards is split into $2p$ piles of $n$ cards, each pile is perfectly shuffled, the piles are stacked atop one another, and $k$ cards are moved from the top of the deck to the bottom.

Consider the position of the ace of spades. If the deck of cards is in a random configuration, then the ace of spades must be equally likely to be in any of the $n$--card piles: $\lbrace 1,2,..,n\rbrace, \lbrace n+1,2,..,2n\rbrace,..., \lbrace n(2p-1)+1,n(2p-1)+2,..,2np\rbrace$. Hence the time it takes until the ace of spades is equally likely to be in any of these piles is a lower bound for the mixing time of the entire chain.

The ace of spades starts in the first pile and with probability $\frac{k}{2n}$ it moves to the right, with probability $\frac{k}{2n}$  it moves to the left and with probability $1- \frac{n}{k}$ it stays in the same pile. This defines a random walk on $\mathbb{Z}/2p\mathbb{Z}$ where the pile $g$ containing the ace of spades evolves according to the following measure:

$$P(g_{t+1}-g_t = \delta)= \begin{cases}
\frac{k}{2n} &\mbox{if } \delta= \pm 1 \\ 
1- \frac{n}{k} & \mbox{if } \delta=0.

\end{cases}$$

This model is analysed in Example $2.1$ of \cite{Threads}. We repeat here the analysis for the lower bound:

Recall that
 $$||P^{\ast t} -U||_{T.V.} =\max_{||f||_{\infty} \leq 1} |\widehat{P}^{\ast t}(f) -U(f)|$$
 
 The test function  $f(j)= e^{\frac{2 \pi j}{2p}}=e^{\frac{\pi j}{p}}$ gives that $\widehat{P}^{\ast t}(f)= \frac{n-k}{n} + \frac{k}{n} \cos \frac{\pi}{p}$.
 Therefore,  for $t = c \frac{2p^2n}{ \pi^2k}$ and using the fact that $\cos x= 1 - \frac{x^2}{2} + O(x^4)$ we have that
 $$||P^{\ast t} -U||_{T.V.} \geq |P^{\ast t}(f) -U(f)|= \left|\frac{n-k}{n} + \frac{k}{n} \cos \frac{ \pi}{p} \right|^{t} \approx \left( 1-\frac{ k \pi^2}{2np^2} \right)^{t} \approx e^{-c}$$ 
 
\begin{remark}
The lower bound presented in this section applies to the lazy version of the walk. To provide a similar lower bound for the non-lazy walk, instead of following the ace of spades one can keep track of the difference of the positions of the initial top two cards. This difference also performs a random walk on the cycle similar to the one described above and so a similar lower bound can be obtained.
\end{remark}
\end{proof} 

\section{Preliminaries for the eigenvalue approach}
\label{sec:prelim}
The proof of Theorem \ref{case n/2} is based on finding the eigenvalues and eigenfunctions of the transition matrix. Spherical function theory will give that the eigenvalues are dual Hahn polynomials. Hence, this section provides necessary preliminaries on spherical function theory and orthogonal polynomials for the proof of Theorem \ref{case n/2}.

\paragraph{Dual Hahn Polynomials:}

Firstly we introduce dual Hahn polynomials, which will be used to prove bounds for the eigenvalues of the Markov chain.

Dual Hahn polynomials are defined as follows: 

\begin{definition}
$$R_k(\lambda(i), n )=  {}_3F_2 \left(  \begin{array}{ccc}
-k,  -i, i-2n-1  \\
   -n,-n
\end{array} ; 1 \right)  := \sum^n_{m=0} \frac{(-k)_{\overline{m}} (-i)_{\overline{m}}  (i-2n-1)_{\overline{m}} }{(-n)_{\overline{m}} (-n)_{\overline{m}}  m!}  $$ $$ \quad  i=0,1,2,...,n$$

\end{definition}

where 
$$\lambda(i)= i (i-2n-1) \mbox{ and } (x)_{\overline{m}} = x(x+1)...(x+m-1) .$$

Thus

$$R_k((\lambda(0), n ))=1, R_k((\lambda(1), n ))=1-\frac{2k}{n} $$
$$\mbox{and}$$
$$ R_k((\lambda(2), n ))=1- \frac{2k(2n-1)}{n^2} + \frac{k(k-1)i(i-1) (2n-1)}{ n^2(n-1)}$$

The dual Hahn polynomials satisfy the difference equation 
\begin{equation}\label{difference}
-k R_k(\lambda(i), n )= B(i)R_k(\lambda(i+1), n ) -(B(i)+D(i)) R_k(\lambda(i), n ) + D(i)R_k(\lambda(i-1), n )
\end{equation}

where $B(i)= \frac{(n-i)(i-2n-1)}{2(2i-2n-1)}$ and $D(i)= \frac{i(i-n-1)}{2(2i-2n-1)}$.

These polynomials also satisfy the orthogonality relation
\begin{equation}\label{orthogonality}
\frac{{2n \choose i} - {2n \choose i-1}}{{2n \choose n}} \sum^n_{y=0} R_y(\lambda(i), n )R_y(\lambda(j), n )  {n \choose y} {n \choose n- y}= \delta_{i,j}
\end{equation}
and the recurrence relation: 
\begin{equation}\label{recurrence}
iR_k(i)= (n-k)^2R_{k+1}(i)-((k-n)^2+k^2)R_k(i) +k^2 R_{k-1}(i)
\end{equation} 

One can find the above relations in paragraph $1.6$ of  Koekoek and Swarttouw \cite{Koekoek}. Other sources are Stanton\cite{Stanton} and Karlin-McGregor \cite{Karlin-McGregor}.

\subsection{A symmetry of the urn model}
\label{sec:sym}

One observes that the operation of switching $k$ balls between the two urns is very similar to that of switching $(n-k)$ balls --- they differ only in that which urn is which has been swapped. In this section, we explore some consequences of this observation for the eigenvalues of the chain. In particular, that the eigenvalues are symmetric with respect to replacing $k$ by $n-k$. 

\begin{lemma}\label{sym_hahn}
Let $\beta_k(i) $ denote the $i^{th}$ eigenvalue of the Bernoulli-Laplace urn model for parameter $k$. Then $$\beta_k(i)=  \begin{cases} - \beta_{n-k}(i) & \mbox{if $i$ is odd} \\ \beta_{n-k}(i) & \mbox{if $i$ is even}  \end{cases}$$
\end{lemma}

\begin{remark}
This is actually a symmetry of the dual Hahn polynomials. After the proof of Lemma \ref{sym_hahn}, we prove a similar result based on this fact.
\end{remark}

\begin{proof}
At first, because of the definition of the dual Hahn polynomials $$\beta_{k}(0)=\beta_{n-k}(0) = 1$$ and $$\beta_{k}(1)= 1 -\frac{2k}{n} = -\beta_{n-k}(1).$$
which satisfy the claim. Equation \ref{difference} and induction will be used to complete the proof. In particular, assume that the claim is true for $k=i$ and for $k=i-1$, then Equation \ref{difference} gives that
\begin{equation}\label{dif}
-k \beta_k(i )= B(i)\beta_k(i+1)) -(B(i)+D(i)) \beta_k(i ) + D(i)\beta_k(i-1 )
\end{equation}
where $B(i)= \frac{(n-i)(i-2n-1)}{2(2i-2n-1)}$ and $D(i)= \frac{i(i-n-1)}{2(2i-2n-1)}$.

If $i$ is odd then (\ref{dif}) becomes
\begin{equation}\label{sec}
k \beta_{n-k}(i )= B(i)\beta_{k}(i+1)) +(B(i)+D(i)) \beta_{n-k}(i ) + D(i)\beta_{n-k}(i-1 )
\end{equation}
Subtracting (\ref{sec}) from (\ref{dif}) for $n-k$ gives that
$$-n \beta_{n-k}(i )= B(i)\left(\beta_{n-k}(i+1) - \beta_{k}(i+1) \right) -2(B(i)+D(i)) \beta_{n-k}(i ) $$
but then 
$$-2(B(i)+D(i)) =-n$$
and therefore 
$$\beta_{n-k}(i+1)) =\beta_{k}(i+1)$$
as desired.

If $i$ is even then (\ref{dif}) becomes
\begin{equation}\label{sec2}
-k \beta_{n-k}(i )= B(i)\beta_{k}(i+1)) -(B(i)+D(i)) \beta_{n-k}(i ) - D(i)\beta_{n-k}(i-1 )
\end{equation}
Adding (\ref{sec2}) to the difference equation (\ref{dif}) for $n-k$ gives 
$$-n \beta_{n-k}(i )= B(i)\left(\beta_{n-k}(i+1) + \beta_{k}(i+1) \right) -2(B(i)+D(i)) \beta_{n-k}(i ) $$
but then 
$$-2(B(i)+D(i)) =-n$$
and therefore 
$$\beta_{n-k}(i+1)) =-\beta_{k}(i+1)$$
as claimed. This completes the proof.
\end{proof}

The following remark connects the symmetry of the eigenvalues given in Lemma \ref{sym_hahn} to the mixing time. It is proved via a different approach to the proof of Lemma \ref{sym_hahn}. 
\begin{remark}\label{sym} 
The walks with parameters $k$ and $n-k$ have the same mixing time. The two walks
have the same eigenvectors and if $\beta_k$ is an eigenvalue of the first walk with respect to the eigenvector
$v$ and $\beta_{n-k}$ is an eigenvalue of the second walk with respect to $v$ then $\beta_{n-k}= \pm \beta_k$.
\end{remark}

\begin{proof}
Denote by $P_k(i,j)$ the probability of starting with $i$ red balls in the right urn and after one step (switching two sets of $k$ balls) ending up with $j$ red balls in the right urn. Then 
$$P_k(i,j)=P_{n-k}(n-i,j)= P_{k}(n-i,n-j)= P_{n-k}(i,n-j)$$

One interpretation of this is that if $A_k=(a_{i,j})$ is the $(n+1) \times (n+1)$ transition matrix of the walk with parameter $k$, then $A_{n-k}= S A_k= A_kS$ where $S$ is the $(n+1) \times (n+1)$ matrix with ones on the antidiagonal and zeros elsewhere \[
   S=
  \left[ {\begin{array}{ccccc}
   0 & 0 & \cdots & 0 & 1 \\
   0 & 0 & \cdots & 1 & 0 \\
   \vdots&  & \ddots & & \vdots \\
   0 & 1 & \cdots & 0 & 0 \\
	 1 & 0 & \cdots & 0 & 0 \\
  \end{array} } \right]
\]

Let \(
   v=
  \left[ {\begin{array}{c}
   v_0  \\
   \vdots \\
   v_n
  \end{array} } \right]
\)
 be a vector, and define \(
\overline{v}  =
  \left[ {\begin{array}{c}
   v_{n}  \\
   \vdots \\
   v_0
  \end{array} } \right]
\) = Sv.

If $v$ is an eigenvector of $A_k$ corresponding to the eigenvalue $\beta_k$ then $A_{n-k}v= SA_k v=
S \beta_k v= \beta_k \overline{v}$ and therefore $(A_{n-k})^2 v= \beta^2_k v$. The fact that both $A_k$ and $A_{n-k}$
are diagonalisable, together with the relation $(A_{n-k})^2 v= \beta^2_k v$ give exactly the result. 
\end{proof}

\begin{remark}
Imitating the above proof shows that the general model where we move $k$ balls from each of $p$ urns (containing $n$ balls each) to the next urn has exactly the same mixing time as the same model where $(n-k)$ balls are moved from each urn to the next.
\end{remark}

\subsection{Eigenvalues and eigenvectors}
\label{sec:eigs}

In Section \ref{sec:closeupper} we will give an upper bound for the 2--pile $k$--cut shuffle when $k=\frac{n}{2} - c$ with $0 \leq c \leq \log_6 n$, using the eigenvectors and eigenvalues of the Markov chain, determined via Gelfand pairs and spherical functions. In this section, we will present necessary background.

Let $G$ be a group and $N$ a subgroup of $G$. If $X=G/N$ then consider the space $L(X)=\lbrace f: X \rightarrow \mathbb{C} \rbrace$.

\begin{definition}
A function $f \in L(X)$ is called \emph{$N$-bi-invariant} if $f(nsr) = f(s)$ for all $s \in G, r,n \in N$.
\end{definition}

\begin{definition}
$(G, N)$ is a \emph{Gelfand pair} if the convolution of $N$-bi-invariant functions is commutative. 
\end{definition}
More details and background on Gelfand pairs can be found in chapter $3$ of \cite{PDbook}. If $(G,N)$ is a Gelfand pair then we have the decomposition
$L(X)= V_1 \oplus V_2 \oplus...\oplus V_m$ where each $(V_i, \rho_i)$ is an irreducible representation of $G$ and no two of the $V_i$ are isomorphic. Moreover, every $V_i$ contains a unique $N$--invariant function $s_i$ with $s_i(id)=1$. The $s_i$ are called spherical functions.

Now, consider an $N$--invariant probability $P$ on $X$. The Fourier transform is defined by
$$\widehat{P}(i)= \sum_{x \in X} s_i(x)P(x)$$ 
and then Lemma $4$ of Chapter $3$ of \cite{PDbook} in the case of our Markov chain says:
$$||P_{id}^t- U||^2 \leq \frac{1}{4} \sum^m_{i=1} d_i |\widehat{P}(i)|^{2t}$$
where $d_i= \dim V_i$.

The above inequality has the following interpretation in terms of the eigenvalues and
eigenvectors of the Markov chain:

\begin{theorem}\label{upper_bound_lemma}
If a reversible Markov chain on a state space $\Omega$ with transition  probability $P(x,y)$ has non-trivial eigenvalues $\beta_j$ and eigenvectors $f_j : \Omega \rightarrow \mathbb{C}$, $j=1,2,...N$ then 
$||P_x^{*t} - \pi_n|| \leq \frac{1}{4}\sum_{j=1} |\beta_j|^{2t} |f_j(x)|^2$.
 
\end{theorem}

The following lemma provides us with the eigenvalues and eigenvectors of the Bernoulli-Laplace Markov chain:

\begin{lemma}
The eigenvalues of the Bernoulli-Laplace Markov chain are the dual Hahn polynomials evaluated at $k$  \begin{equation}\label{Hahn}
s_r(k)=\sum^r_{m=0} \frac{(-r)_{\overline{m}} (r-2n-1)_{\overline{m}} (-k)_{\overline{m}}}{(-n)^2_{\overline{m}} m!}
\end{equation}
 and the eigenvectors are the normalised dual Hahn polynomials  $$
s_r(x)=  \frac{\sqrt{{2n \choose r} - {2n \choose r-1}}}{{n \choose \frac{n}{2}}{\frac{n}{2} \choose x}^2} \sum^r_{m=0} \frac{(-r)_{\overline{m}} (r-2n-1)_{\overline{m}} (-x)_{\overline{m}}}{(-n)^2_{\overline{m}} m!}
$$
 for each $r=0,1,2,...,n$, where 
$(y)_{\overline{m}} =y(y+1)...(y+m-1)$.

\end{lemma}

\begin{proof}
Chapter $3$ of \cite{PDbook} states that $(S_{2n}, S_n \times S_n)$ is a Gelfand pair and therefore all we need to do is to find the spherical functions. Also, James (\cite{James}, pg.52) proved that $L(X)= V_1 \oplus V_2 \oplus...\oplus V_m$ where each $(V_i, \rho_i) $ corresponds to the partition $(n-i,i)$. Therefore $L(X)$ breaks down into $\frac{n}{2}$ such representations. 

The spherical functions for the Gelfand pair $(S_{2n}, S_n \times S_n)$ were determined by Karlin and McGregor in \cite{Karlin-McGregor}. In \cite{Stanton} Stanton expresses this result in modern terminology. The spherical functions are exactly the dual Hahn polynomials:
$$
 s_r(x)=\sum^r_{m=0} \frac{(-r)_m (r-2n-1)_m (-x)_m}{(-n)^2_m m!}
$$
  The eigenvalues of the Markov chain are given by the Fourier transform:
 $$\widehat{P}(i)= \sum_{x \in X} s_i(x)P(x)= s_i(k) \sum_{x \in X} P(x)= s_i(k)$$
 where the $s_i(k)$ can be factored out because of its  $N$-invariance. 

Letac \cite{Letac} explains how if $(G,N)$ is a Gelfand pair and $P$ is an $N-$invariant probability measure on $G$ then $\widehat{P}(\rho)$ is diagonalizable by the spherical functions. Theorem $6$ of \cite{PDbook} (page $49$) and the remark afterwards also confirm that for the Markov chain studied in this chapter, the eigenfunction corresponding to the eigenvalue $s_i(k) $ is exactly $s_i(x)$.  

\end{proof}

\section{The case of two piles}\label{two p}
\subsection{Upper bounds for $k$ close to $\frac{n}{2}$}\label{sec:closeupper}

In this section, we prove Theorem \ref{case n/2}, an upper bound for the mixing time of the two pile $k$--cut shuffle when $k$ is very close to $\frac{n}{2}$ --- that is, almost half of the balls are being moved between the two urns at each step.

We first explicitly find the eigenvalues of the walk for the cases $k=\frac{n}{2}$ and $k=\frac{n}{2}-1$.

\begin{lemma}
When $k= \frac{n}{2}$, the eigenvalues are
$$\beta_i \left( \frac{n}{2} \right)= \begin{cases} 0 & \mbox{if $i$ is odd} \\ \frac{(i-1)(i-3)(i-5)...1 (n+2-i)(n+4-i)...n }{(n-i+1)(n-i+3)...(n-1)(i-2n-2)(i-2n-4)...(-2n)} & \mbox{if $i$ is even}  \end{cases}$$
\end{lemma}

\begin{proof}
Since the eigenvalues of the walk are the dual Hahn polynomials, they satisfy the following difference relation (Equation \ref{difference}):
$$-k\beta_i(k,n)= B(i) \beta_{i+1}(k,n) - (B(i)+D(i))\beta_i(k,n) + D(i) \beta_{i-1}(k,n)$$
where $B(i)= \frac{(n-i)(i-2n-1)}{2(2i-2n-1)}$ and $D(i)= \frac{i(i-n-1)}{2(2i-2n-1)}$.
Therefore \begin{multline*}
(-2k(2i-2n-1)+ (n-i)(i-2n-1)+ i(i-n-1))\beta_i(k,n) \\ = (n-i)(i-2n-1) \beta_{i+1}(k,n) + i(i-n-1) \beta_{i-1}(k,n) \end{multline*}
For $k=\frac{n}{2}$ the left hand side is equal to zero and so
$$\beta_{i+1}(k,n)= \frac{i(n-i+1)}{(n-i)(i-2n-1)} \beta_{i-1}(k,n) $$
For general $k$, the first eigenvalue is $\beta_1(k,n)= 1-\frac{2k}{n}$ so  $\beta_1(\frac{n}{2},n)=0$. This gives that the
odd eigenvalues are all equal to zero. 

At the same time, $\beta_0(k,n)=1$ and so induction gives that $$\beta_i=\frac{(i-1)(i-3)(i-5)\cdots1\cdot(n+2-i)(n+4-i)\cdots n }{(n-i+1)(n-i+3)\cdots(n-1)(i-2n-2)(i-2n-4)\cdots(-2n)}$$ in the case where $i$ is even.

\end{proof}

\begin{lemma}\label{eig_bound}
The eigenvalues of the walk for $k=\frac{n}{2}$ satisfy $|\beta_i| \leq \frac{1}{n}$. 
\end{lemma}

\begin{proof}
First of all, if $n$ is even then $$n |\beta_n |= \frac{2\cdot 4\cdots(n-1)n\cdot n}{(n+2)(n+4)\cdots(2n)} = \frac{4\cdot 6\cdots n}{(n+2)\cdots(2n-2)} \leq 1$$ since $j \leq n+j-2$. 

If $n$ is odd then 
\begin{align*}
n|\beta_{n-1}|& = \frac{3 \cdot 5\cdots(n-4)(n-2)3 \cdot 5\cdots(n-2)n\cdot n}{2 \cdot 4 \cdots(n-3)(n-1)(n+3)(n+5)\cdots(2n-2)2n}\\
&= \frac{3\cdot 5\cdots(n-2)3\cdot 5\cdots(n-2)n}{4 \cdot 6\cdots(n-1) 4(n+3)\cdots(2n-2)}\\
& \leq 1.\\
\end{align*}
 
 because $j\leq j+1$, $3 \leq 4$ and $j\leq n+j-2$. 
 
 Moreover, $|\beta_2|= \frac{1}{2(n-1)} \leq \frac{1}{n}$.
 
 We now consider the cases $2 < i \leq n-2$. 
 
 We would like to show that the quantity 
	
 $$n|\beta_i|=\frac{(i-1)(i-3)(i-5)\cdots1\cdot(n+2-i)(n+4-i)\cdots n\cdot n}{(n-i+1)(n-i+3)\cdots(n-1)(2n+2-i)(2n+4-i)\cdots(2n)} $$ is at most one. This is equivalent to 
 $$
  \frac{(i-1)(i-3)(i-5)\cdots1\cdot(n+2-i)(n+4-i)\cdots n }{(n-i+1)(n-i+3)\cdots(n-3)2(n-1)(2n+2-i)(2n+4-i)\cdots(2n-2)}  \leq 1,$$
  which is true because
  $n+j-i \leq 2n +j-i$ for $2 \leq j \leq i-2$, $n \leq 2(n-1)$ and $i-j \leq n-j-2$ for $1 \leq j \leq i-3$.
\end{proof}

The next theorem shows that for $k=\frac{n}{2}$, $t=3$ steps are enough to ensure mixing. 
\begin{reptheorem}{case n/2}(Case $k=\frac{n}{2}$)

For the case $k=\frac{n}{2}$ and  $t \geq 3$ then
$$4||P_0^{*t} - \pi_n||^2_{T.V.} \leq   \frac{\pi^2}{6 n^{2t-4}}$$
\end{reptheorem}

\begin{proof}
By Theorem \ref{upper_bound_lemma}, because the starting configuration has zero red balls in the right urn,
the total variation distance is bounded by:
\begin{align*}
4||P_0^{*t} - \pi_n||^2_{T.V.} & \leq\sum^{n}_{i=1} |\beta_i|^{2t} |s_i(0)|^2 \\
& = \sum^n_{i=2\mbox{, $i$ even} } \left(\frac{(i-1)(i-3)(i-5)\cdots1\cdot(n+2-i)(n+4-i)\cdots n }{(n-i+1)(n-i+3)\cdots(n-1)(i-2n-2)(i-2n-4)\cdots(-2n)}\right)^{2 t}\\
& \hspace{10cm} \cdot \left({2n \choose i} - {2n \choose i-1}\right)\\ 
& \leq  \frac{\pi^2}{6 n^{2t-4}}\\
\end{align*}

The last step is true because if $t \geq 3$ then $$\left(\frac{(i-1)(i-3)(i-5)\cdots1\cdot(n+2-i)(n+4-i)\cdots n }{(n-i+1)(n-i+3)\cdots(n-1)(i-2n-2)(i-2n-4)\cdots(-2n)}\right)^{2t} \cdot \left({2n \choose i} - {2n \choose i-1} \right) \leq \frac{1}{n^{2t-4}}.$$ In more detail, $${2n \choose i} - {2n \choose i-1}= \frac{(2n)! (2n-2i+1)}{i!(2n-i+1)!}$$

which allows us to calculate as follows 
\begin{align*}
&\left(\frac{(i-1)(i-3)(i-5)\cdots1\cdot(n+2-i)(n+4-i)\cdots n }{(n-i+1)(n-i+3)\cdots(n-1)(i-2n-2)(i-2n-4)\cdots(-2n)}\right)^{2t} \cdot \left({2n \choose i} - {2n \choose i-1}\right) \\
& = \left(\frac{(i-1)(i-3)(i-5)\cdots1\cdot(n+2-i)(n+4-i)\cdots n }{(n-i+1)(n-i+3)\cdots(n-1)(2n+2-i)(2n+4-i)\cdots(2n)}\right)^{2t} \cdot \left(\frac{(2n)! (2n-2i+1)}{i!(2n-i+1)!}\right)\\
&\leq 
 \left(\frac{(i-1)(i-3)(i-5)\cdots1\cdot(n+2-i)(n+4-i)\cdots n }{(n-i+1)(n-i+3)\cdots(n-1)(2n+2-i)(2n+4-i)\cdots(2n)}\right)^{2t} \\
&\hspace{7cm}\cdot\left(\frac{(2n-i+2)(2n-i+3)\cdots(2n) (2n-2i+1)}{i!}\right) \\ 
& \leq \left(\frac{ (n+2-i)(n+4-i)\cdots n }{(n-i+1)(n-i+3)\cdots(n-1)}\right)^{2t} \left(\frac{(i-1)(i-3)(i-5)\cdots1\cdot }{(2n+2-i)(2n+4-i)\cdots(2n)}\right)^{2t-1} \\ &\hspace{7cm}\cdot\left(\frac{(2n-i+1)(2n-i+3)\cdots(2n-1)}{i(i-2)\cdots2}\right) \\
& \leq \left(\frac{ (n+2-i)(n+4-i)\cdots n }{(n-i+1)(n-i+3)\cdots(n-1)}\right)^{2t} \left(\frac{(i-1)(i-3)(i-5)\cdots1\cdot }{(2n+2-i)(2n+4-i)\cdots(2n)}\right)^{2t-2} \\
& \leq \beta^{2t-2}_i \left(\frac{(n+2-i)(n+4-i)\cdots n}{(n+1-i)(n+3-i)\cdots(n-1)}\right)^2 \\
& \leq \beta^{2t-2}_i  \frac{n^2}{(n-i+1)^2} \\
\end{align*}
and using Lemma \ref{eig_bound} we get that 
$$4||P_0^{*t} - \pi_n||^2_{T.V.} \leq \sum^n_{i=2, i \mbox{ even}} \frac{1}{n^{2t-2}} \frac{n^2}{(n-i+1)^2}
\leq  \frac{\pi^2}{6 n^{2t-4}}.$$

\end{proof}

We will now show that the mixing time for $k=\frac{n}{2}-1$ is also finite and independent of $n$.

\begin{lemma}
If $n$ is even and $k= \frac{n}{2}-1$ then the eigenvalues are:
$$\beta_i\left(\frac{n}{2}-1\right)= \begin{cases} -\frac{2}{n(2n-1)} & \mbox{if } i=1\\
\frac{2i(i-2)\cdots1\cdot(n+3-i)(n+5-i)\cdots(n-2)(n+1-i)}{(2i-2n-1)n(n-i+2)(n-i+4)\cdots(n-1)(i-2n-3)(i-2n-5)\cdots(-2n+2)(-2n)} & \mbox{if } i \mbox{ is odd} \\ \frac{(2i+n^2)(i-1)(i-3)(i-5)\cdots1\cdot(n+2-i)(n+4-i)\cdots n }{n^2(n-i+1)(n-i+3)\cdots(n-1)(i-2n-2)(i-2n-4)\cdots(-2n)} & \mbox{if } i \mbox{ is even}  \end{cases}$$

Each of these eigenvalues satisfies $|\beta_i\left(\frac{n}{2}-1\right)| \leq \frac{1}{n}$.
\end{lemma}

\begin{proof}
$\\$
Using the recurrence relation that the dual Hahn polynomials satisfy for $k=\frac{n}{2}$, one can find the 
eigenvalues for $k=\frac{n}{2}-1$. Namely for $k=\frac{n}{2}$ the relation 

$$i \beta_k (i)= (n-k)^2\beta_{k+1}(i)-((k-n)^2+k^2)\beta_k(i) +k^2 \beta_{k-1}(i)$$ gives that

$$i \beta_{\frac{n}{2}}(i)= \frac{n^2}{4}\beta_{\frac{n}{2}+1}(i)-\frac{n^2}{2}\beta_k(i) +\frac{n^2}{4}2 \beta_{\frac{n}{2}-1}(i).$$

The fact that $\beta_{\frac{n}{2}}(i)=0$ for odd $i$ implies that $\beta_{\frac{n}{2}-1}(i)=- \beta_{\frac{n}{2}+1}(i)$, while the fact that
$\beta_{\frac{n}{2}-1}(i) \neq 0$ in combination with remark \ref{sym} imply that for even $i$
$\beta_{\frac{n}{2}-1}(i)=  \beta_{\frac{n}{2}+1}(i) $ and therefore 

\begin{align*}
\beta_{\frac{n}{2}-1}(i) & = \frac{2(i + \frac{n^2}{2})}{n^2} \beta_{\frac{n}{2}}(i) \\
 & =\frac{(2i+n^2)(i-1)(i-3)(i-5)\cdots1\cdot(n+2-i)(n+4-i)\cdots n }{n^2(n-i+1)(n-i+3)\cdots(n-1)(i-2n-2)(i-2n-4)\cdots(-2n)}.\\ \end{align*}

Now we use the difference equation (Equation \ref{difference}) to get the odd eigenvalues. Namely,
the relation
\begin{multline*}(-2k(2i-2n-1)+ (n-i)(i-2n-1)+ i(i-n-1))\beta_i(k,n)\\=(n-i)(i-2n-1) \beta_{i+1}(k,n) + i(i-n-1) \beta_{i-1}(k,n)\end{multline*}

gives that for $k=\frac{n}{2}-1$ and $i \geq 3$ odd the eigenvalues are

$$\beta_i= \frac{2i(i-2)\cdots1\cdot(n+3-i)(n+5-i)\cdots(n-2)(n+1-i)}{(2i-2n-1)n(n-i+2)(n-i+4)\cdots(n-1)(i-2n-3)(i-2n-5)\cdots(-2n+2)(-2n)}.$$

The value of $\beta_1$ is obtained by substituting $i=1$ and $k=\frac{n}{2}-1$ into Equation \ref{Hahn}.

The proof of the fact that for $k = \frac{n}{2}-1$, each eigenvalue satisfies $|\beta_i\left(\frac{n}{2}-1\right)| \leq \frac{1}{n}$ is similar to the case $k=\frac{n}{2}$.
\end{proof}

\begin{reptheorem}{case n/2}(Case $k=\frac{n}{2}-1$)
$\\$
For the case $k=\frac{n}{2}-1$ we have that
$$4||P_0^{*t} - \pi_n||^2_{T.V.} \leq   \frac{\pi^2}{3 n^{2t-4}}$$
\end{reptheorem}

\begin{proof}
$\\$
Since $|\beta_i| \leq \frac{1}{n}$ and 
$$|\beta_i|^{2t}  \left({2n \choose i} - {2n \choose i-1} \right) \leq 2 \frac{n^2}{(n-i+1)^2}
|\beta_i|^{2t-2},$$
we may imitate the proof of Theorem \ref{case n/2} for the case $k=n/2$ to get that
$$4||P_0^{*t} - \pi_n||^2_{T.V.} \leq \sum^n_{i=1} \frac{1}{n^{2t-2}} \frac{2n^2}{(n-i+1)^2}
\leq  \frac{\pi^2}{3 n^{2t-4}}$$
\end{proof}

The proofs for the cases $k=\frac{n}{2}$ and $k= \frac{n}{2}-1$ generalise to the proof of the following theorem regarding cases when $k$ is close to $\frac{n}{2}$.

\begin{reptheorem}{case n/2}(Case $k=\frac{n}{2}-c$, with $ 2 \leq c < 1+ \log_6 n$)

If $n$ is even and $k=\frac{n}{2} - c$, where $2 \leq c < 1+ \log_6 n$  then
$$||P_0^{*t} - \pi_n||_{T.V.} \leq 12 \pi^2 n^2 \left(\frac{6^{c-1} }{n} \right)^{2t-2}$$
\end{reptheorem}

\begin{proof}
$\\$
We proceed by induction on $c$.  As shown above for $c=0$ or $c=1$ the eigenvalues satisfy the relations
$$|\beta_k(i)| \leq \frac{1}{n}$$ and $$|\beta_k(i)|^{2t}  \left({2n \choose i} - {2n \choose i-1} \right) \leq 2 \frac{n^2}{(n-i+1)^2}
|\beta_k(i)|^{2t-2}.$$

Now assume that there is some $A>0$ such that for $k=\frac{n}{2} -c$ or $k= \frac{n}{2} -(c+1)$
$$|b_k(i)| \leq \frac{A}{n}$$ and 
$$|\beta_k(i)|^{2t}  \left({2n \choose i} - {2n \choose i-1} \right) \leq 2 B\frac{n^2}{(n-i+1)^2}
|\beta_k(i)|^{2t-2} $$
We will show by induction that $$|b_{\frac{n}{2} -(c+2)}| \leq \frac{6A}{n}$$ and 
$$|\beta_{\frac{n}{2} -(c+2)}(i)|^{2t}  \left({2n \choose i} - {2n \choose i-1} \right) \leq 36 B\frac{2n^2}{(n-i+1)^2}
|\beta_{\frac{n}{2} -(c+2)}(i)|^{2t-2} .$$

For the inductive step, the recurrence relation between dual Hahn polynomials (Equation \ref{recurrence}) gives us that
 
 $$\beta_{k-2}(i)= \frac{i+(k-1)^2+ (k-1-n)^2}{(k-1)^2} \beta_{k-1}(i) - \frac{(k-1-n)^2}{(k-1)^2} \beta_k(i) $$
 
So for $k= \frac{n}{2} - c$ where $0< c \leq \log_6 n$ we have that 

\begin{equation}\label{1}
\frac{i+(k-1)^2+ (k-1-n)^2}{(k-1)^2} \leq 4
\end{equation}
and 
\begin{equation}\label{2}
\frac{(k-1-n)^2}{(k-1)^2} \leq 2
\end{equation}
Notice that for (\ref{2}) it is crucial that $k$ is of order $n$. Finally,
In total,
$$|\beta_{\frac{n}{2}-c-2}(i)| \leq 6 \max \{|\beta_{\frac{n}{2}-c-1}(i)|, |\beta_{\frac{n}{2}-c}(i)| \}$$
which implies that
$$|\beta_{\frac{n}{2}-c-2}(i)| \leq \frac{6A}{n} $$
and
$$|\beta_{\frac{n}{2} -(c+2)}(i)|^{2t}  \left({2n \choose i} - {2n \choose i-1} \right) \leq 36 \frac{2n^2}{(n-i+1)^2}
|\beta_{\frac{n}{2} -(c+2)}(i)|^{2t-2}.$$
Specifically, $|\beta_{\frac{n}{2}-c}(i)| \leq \frac{6^{c-1}}{n}$ if $2 \leq c < 1+ \log_6 n$.
The remainder of the proof is similar to the proof of Theorem \ref{case n/2} for $k=n/2$. 

\end{proof}

\subsection{A coupling bound}\label{sec:coupling2}

In this subsection, we introduce a coupling argument which will lead us to an upper bound for general $k$.
We are still considering the walk on two urns, each containing $n$ balls, from a set of $n$ balls from each of two colours. At each step, we exchange randomly chosen sets of $k$ balls from each urn.
\paragraph{The coupling}
$\\$
Consider two instances of this chain, the first in state $S_a$ and the second in state $S_b$. Without loss of generality, let $a > b$.

We couple the two chains as follows. In each chain, label the balls in the left urn by the numbers $1, 2, \dots, n$, and the balls in the right urn by the numbers $n+1, n+2, \dots, 2n$, so that any red ball has a lower number than any white ball in the same urn.

That is, the contents of the urns are as follows:

\begin{tabular}{llllll}
Chain & Urn & \# Red & \# White & Red labels & White labels \\
1 & Left & $a$ & $n-a$ & 1 to $a$ & $(a+1)$ to $n$ \\
1 & Right & $n-a$ & $a$ & $n+1$ to $2n-a$ & $(2n-a+1)$ to $2n$ \\
2 & Left & $b$ & $n-b$ & 1 to $b$ & $(b+1)$ to $n$ \\
2 & Right & $n-b$ & $b$ & $n+1$ to $2n-b$ & $(2n-b+1)$ to $2n$ \\
\end{tabular}

\begin{figure}[ht!]
\centering
\includegraphics[scale=0.7]{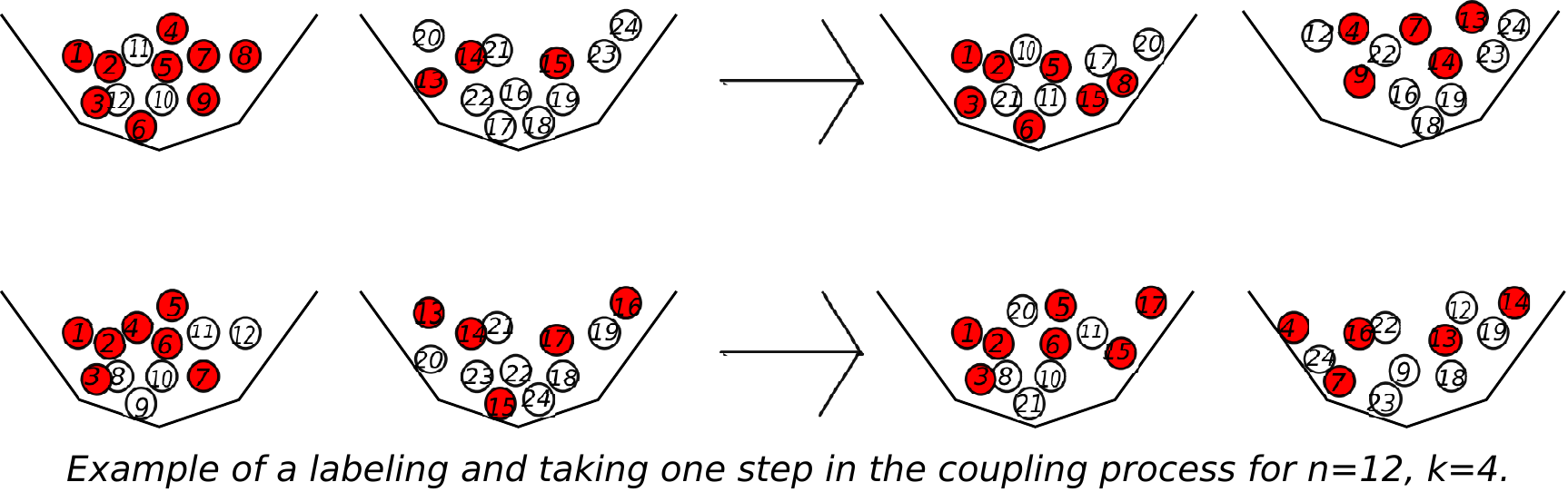}
\\
\end{figure}

Now, choose $A$ to be a $k$-element subset of $\{1,2,\dots,n\}$, and  $B$ to be a $k$-element subset of $\{n+1,n+2,\dots,2n\}$, chosen uniformly at random. In each chain, simultaneously move the balls with labels in $A$ from the left urn to the right urn, and move the balls with labels in $B$ from the right urn to the left urn. Then remove all labels.

Restricted to either of the chains, this is the process as described initially, because any 
$k$-element subsets of the balls in the left and right urns can be chosen, with uniform probability. In later sections, we will refer to this coupling as $C$. We will now analyse the impact of the coupling $C$.

Consider our two chains, initially in states $S_a$ and $S_b$, with $a > b$. Let $c = a-b$ be difference between the number of red balls in the left urn in the two chains.

For each element of the set $A$, a ball is moved from the left urn to the right urn in each chain. If these balls are both red or both white, then there is no change to $c$ --- in each chain, the change to the number of red balls in the left urn is the same. Our coupling does not allow a white ball to be moved in the first chain and a red ball in the second chain. The only remaining possibility is that a red ball is moved in the first chain and a white in the second. This decreases $c$ by one.

Likewise, moving balls from the right urn to the left urn either leaves $c$ unchanged (if the two balls are the same colour), or decreases $c$ by one (if a white ball is moved in the first chain and a red in the second).

The maximum amount by which $c$ can be decreased is $2(a-b)$, when all of the mismatched balls are moved between urns. That is, when $A$ includes each integer between $b+1$ and $a$ and $B$ includes each integer between $n+b+1$ and $n+a$.

Therefore, after our exchange of $k$ balls from each urn, the difference between the number of red balls in the left urn of each chain is between $c$ and $-c$. Without loss of generality, assume that this difference is nonnegative (else we swap the two colours).

Thus, the distance between the states of the two chains does not increase.

\paragraph{Path Coupling}$\\$
We will now present a path coupling argument, a technique introduced by Bubley and Dyer in \cite{BubleyDyer}. In order to use path coupling, we need to specify a graph structure on $\S$, the set of states of the system, and specify lengths for the edges of this graph. For now, let us say that each $S_a$ is adjacent to $S_{a-1}$ and to $S_{a+1}$, with these edges having length $1$. This metric is the same as that defined earlier, $d(S_a,S_b) = |a-b|$.

The path coupling technique requires us to consider a pair of systems in any two adjacent 
states, $S_a$ and $S_{a+1}$. If we can find a coupling $C$ between the two systems and a 
constant $\alpha > 0$ so that when the systems take one step according to the coupling 
$C$, the expected distance between the two resulting states is at most 
$e^{-\alpha}d(S_a,S_{a+1})$, then the mixing time is bounded by 
\begin{equation}\label{mix}
t_{\text{mix}}(\epsilon) \leq 
\frac{-\log(\epsilon) + \log(\diam(S))}{\alpha}
\end{equation}
Equation \ref{mix} will lead to the proof of Theorem \ref{general two piles upper bound}:
\begin{proof}
Using the coupling $C$ defined in the previous section, let us calculate the expected distance between the two chains after one step according to $C$.

Consider the difference $c$ between the number of red balls in the left urn between the two chains. Initially, $c=1$. There is a probability of $\frac{k}{n}$ that $c$ decreases by one due to a mismatched pair of balls being moved from the left urn to the right, and an independent probability of $\frac{k}{n}$ that $c$ decreases by one due to a mismatched pair of balls being moved from the right urn to the left. (There is only one mismatched pair on each side, and there is a probability of $\frac{k}{n}$ that the appropriate label is chosen to be included in $A$ or $B$).

If $c$ decreases by two, then we relabel the colours, and so the distance between the two chains remains one. If $c$ decreases by exactly one, then the two chains are now in the same state. If $c$ does not decrease at all, then the distance remains one.

The following table displays the possible outcomes.

\begin{tabular}{lll}
Change in $c$ & Probability & Final distance \\
$0$ & $\left(\frac{n-k}{n}\right)^2$ & 1 \\
$-1$ & $2\left(\frac{k}{n}\right)\left(\frac{n-k}{n}\right)$ & 0 \\
$-2$ & $\left(\frac{k}{n}\right)^2$ & 1 \\
\end{tabular}

Thus, if $(X_t)_{t}, (Y_t)_t$ are two copies of the Markov chain that are in adjacent sites at time $t-1$, the expected distance between the two chains after a single step according to the coupling $C$ is 
\begin{equation}\label{pc}
E(d(X_t,Y_t)) = \left(\frac{n-k}{n}\right)^2 + \left(\frac{k}{n}\right)^2= \left( 1- \frac{2k(n-k)}{n^2}\right)
\end{equation}

Therefore, we have that $\alpha =- \log \left( 1- \frac{2k(n-k)}{n^2} \right).$
The diameter of the graph $\S$ is $n$, so Equation \ref{mix} gives that the mixing time is bounded by $$t_{\text{mix}}(\epsilon) \leq \frac{-\log (\epsilon) + \log (n)}{- \log \left( 1- \frac{2k(n-k)}{n^2} \right)} \approx \frac{n^2}{2k(n-k)} \log (n)  -  \frac{n^2}{2k(n-k)} \log (\epsilon)$$

If $\lim_{n \longrightarrow \infty}\frac{k}{n} = 0$, then this bound is asymptotically $$t_{\text{mix}}(\epsilon) \leq \frac{n}{2k}\log \left(\frac{n}{\epsilon} \right)$$
while if $\lim_{n \longrightarrow \infty}\frac{k}{n} = b \leq \frac{n}{2}$, this bound is asymptotically $$t_{\text{mix}}(\epsilon) \leq \frac{1}{2b(1-b)}\log \left(\frac{n}{\epsilon} \right)$$
\end{proof}

\subsection{Lower bounds for the two urn model}
\label{sec:lower}
This section presents the proof of Theorem \ref{lower two piles}.
\begin{proof}
Remember that the definition of the total variation distance is
$$||P-Q||_{T.V.} = \sup_{A} |P(A)-Q(A)|.$$

Any choice of a specific set $A$ will provide a lower bound.  To find a suitable set $A$ and prove Theorem \ref{lower two piles}, consider the normalised dual Hahn polynomial of degree one $f(x)= \sqrt{n-1} (1-\frac{2x}{n}).$
Then, consider $A_{\alpha} = \{ x: |f(x)| \leq \alpha \}.$ A specific choice for $\alpha$
will give the correct lower bound, using the second moment method. 

The orthogonality relations between normalised dual Hahn polynomials, Equation \ref{orthogonality}, give that
if $Z$ is a point chosen uniformly in $X=\{0,1,2,\dots,n\}$ then
$$ \mathbb{E}(f(Z))=0 \mbox{ and } \Var(f(Z))=1$$
Now under the convolution measure, 
$$ E \left( f(Z_t) \right) = \sqrt{n-1} \left( 1- \frac{2k}{n} \right)^{t}$$
because $f$ is an eigenfunction of the Markov chain corresponding to the eigenvalue $1-\frac{2k}{n}.$

Thus, under the convolution measure

$$Var \{ f(Z_t) \} =\frac{n-1}{2n-1} + \frac{(n-1)(2n-2)}{(2n-1)} \left( 1- \frac{2k(2n-1)}{n^2} + \frac{2k(k-1)(2n-1)}{n^2 (n-1)}\right)^t - (n-1) \left( 1-\frac{2k}{n} \right)^{2t}.$$

To justify this, remember that the first three (non-normalised)  eigenfunctions of this Markov chain are:
$$f_0(x)=1, f_1(x)=1-\frac{2x}{n} \mbox{ and }f_2(x)= 1 - \frac{2x(2n-1)}{n^2} +\frac{2x(x-1)(2n-1)}{n^2 (n-1)}.$$
To proceed to the analysis, consider the following cases:

Firstly,  $f^2_1(x)= \frac{1}{2n-1} f_0(x) + \frac{2n-2}{2n-1} f_2(x)$. Combining this and the fact that $f_2$ corresponds to the eigenvalue $1- \frac{2k(2n-1)}{n^2} + \frac{2k(k-1)(2n-1)}{n^2 (n-1)}$ gives the above expression for the variance.

\begin{enumerate}
\item  If $\lim_{n \rightarrow \infty} \frac{k}{n}= d \in (0,\frac{1}{2})$ we can write that for $t=\frac{1}{2} \log_{1-2d} (n) -c$, asymptotically
$$ E \{ f(Z_t)\} = \sqrt{n-1} (1-2d)^t= (1-2d)^c$$
and 
$$\Var \{ f(Z_t) \} \leq \frac{3}{2}$$
 
\item If $\lim_{n \rightarrow \infty} \frac{k}{n}=0$ then for $t$ of the form $\frac{1}{4k} n \log (n) - \frac{cn}{k} $ (where $c>0$) the mean  becomes
$$E \{ f(Z_t ) \}= \exp \left( 2c + O \left( \frac{k\log n}{n} \right) +O \left(\frac{ck}{n} \right)  \right) $$
and the variance for large $n$ becomes 
\begin{align*}
\Var \{ f(Z_t) \} & \approx  \frac{1}{2} + n \left(1-\frac{4k}{n} \right)^t  -(n-1) \left(1- \frac{2k}{n}\right)^{2t} \\
& \leq \frac{1}{2} + n \left(1-\frac{4k}{n} \right)^t  -(n-1) \left(1- \frac{4k}{n}\right)^{t}\\
& \approx \frac{1}{2} + \frac{\exp(4c)}{n}.
\end{align*}
%O\left(\frac{1}{n}\right) +\exp \left( 4c +  O\left(\frac{k\log n}{n}\right) +O\left(\frac{ck}{n}\right)  \right) - \exp \left( 4c + O\left(\frac{k\log n}{n}\right) +O\left(\frac{ck}{n}\right)  \right)\\
%& = \frac{1}{2}+ e^{4c} \left( O\left(\frac{k\log n}{n}\right) +O\left(\frac{ck}{n}\right)  \right)
Therefore the variance is uniformly bounded for $0 \leq c < \frac{1}{4} \log n$.
\end{enumerate}
In both cases Chebyshev's inequality will give that for the set $A_{\alpha} = \{ x: |f(x)| \leq \alpha \}$,
$$\pi_n(A_{\alpha} ) \geq 1- \frac{1}{\alpha^2} \mbox{ while } P^{*t}(A_{\alpha}) < \frac{B}{(e^{2c} - \alpha)^2} $$ where B is uniformly bounded when $0 \leq c < \frac{1}{4} \log n$.
Therefore, 

$$||P^{*t} - \pi_n||_{T.V.} \geq 1- \frac{1}{\alpha^2} - \frac{B}{(e^{2c}- \alpha)^2}.$$

Taking $\alpha= \frac{e^{2c}}{2}$ completes the proof.
\end{proof}

\paragraph*{Remarks}
\begin{enumerate}
\item The above proof doesn't hold when $k-\frac{n}{2}$ is $O(1)$. This is because $E(f(Z_t)) \approx 0$ (independently of $t$) and therefore  Chebyshev's inequality would give that
$\pi_n(A_{\alpha})$ is very close to 1. 
\item Our results for the $2$ pile case may seem surprising. In particular, if it takes a constant number of steps with $k=\frac{n}{2}$ to randomise the deck, then one might think that two steps with $k=\frac{n}{4}$ should be about as good as one step with $k=\frac{n}{2}$, and so that a (larger) constant number of steps should suffice for $k=\frac{n}{4}$. The lower bound for $k=\frac{n}{4}$, however, is proportional to $\log n$.

The error in this logic is that two $\frac{n}{4}$--steps are not as good as one $\frac{n}{2}$--step, and indeed, the contribution from repeated $\frac{n}{4}$--steps decreases as more of them are used. This can be seen in the stationary distribution, which in our case is the hypergeometric distribution with mean $\frac{n}{2}$ and standard deviation $\sqrt{\frac{n}{8}}$. Assume we are very close to the stationary measure. This means that  the number of red balls in the right urn is very likely to be in the interval $\left( \frac{n}{2} - \sqrt{\frac{n}{8}}, \frac{n}{2} + \sqrt{\frac{n}{8}}\right)$. Now if we move $\frac{n}{2}$ balls from each urn we are in this interval immediately and we stay there, while when we move $\frac{n}{4}$ balls we need $\log n$ steps to get there. 
\end{enumerate}

\section{Acknowledgements}
We would like to thank Persi Diaconis for suggesting this problem, as well as for many helpful discussions and comments.

\bibliographystyle{plain}
\bibliography{Urns}

\def\cprime{$'$}
\begin{thebibliography}{10}

\bibitem{Bernoulli_site}
Daniel {I} {B}ernoulli.
\newblock {\em http://cerebro.xu.edu/math/Sources/DanBernoulli/index.htm}.

\bibitem{Bayer}
Dave Bayer and Persi Diaconis.
\newblock Trailing the dovetail shuffle to its lair.
\newblock {\em Ann. Appl. Probab.}, 2(2):294--313, 1992.

\bibitem{D.Bernoulli_three}
Daniel Bernoulli.
\newblock Disquisitiones analyticae de novo problemate coniecturali.
\newblock {\em Novi Commentarii Acad. Petrop}, XIV:3--25, 1769.

\bibitem{BubleyDyer}
Russ Bubley, Martin Dyer, and Leeds~Ls Jt.
\newblock Path coupling: A technique for proving rapid mixing in markov chains.
\newblock In {\em In FOCS ’97: Proceedings of the 38th Annual Symposium on
  Foundations of Computer Science (FOCS}, page 223, 1997.

\bibitem{PDbook}
Persi Diaconis.
\newblock {\em Group representations in probability and statistics}.
\newblock Institute of Mathematical Statistics Lecture Notes---Monograph
  Series, 11. Institute of Mathematical Statistics, Hayward, CA, 1988.

\bibitem{Threads}
Persi Diaconis.
\newblock Threads through group theory.
\newblock In {\em Character theory of finite groups}, volume 524 of {\em
  Contemp. Math.}, pages 33--47. Amer. Math. Soc., Providence, RI, 2010.

\bibitem{Bernoulli-Laplace}
Persi Diaconis and Mehrdad Shahshahani.
\newblock Time to reach stationarity in the {B}ernoulli-{L}aplace diffusion
  model.
\newblock {\em SIAM J. Math. Anal.}, 18(1):208--218, 1987.

\bibitem{Feller}
William Feller.
\newblock {\em An introduction to probability theory and its applications.
  {V}ol. {I}}.
\newblock Third edition. John Wiley \& Sons, Inc., New York-London-Sydney,
  1968.

\bibitem{MJ}
Martin Jacobsen.
\newblock Laplace and the origin of the {O}rnstein-{U}hlenbeck process.
\newblock {\em Bernoulli}, 2(3):271--286, 1996.

\bibitem{James}
G.D. James.
\newblock {\em Representations of the symmetric group}.
\newblock Lecture notes. Springer-Verlg, Berlin, 1978.

\bibitem{Karlin-McGregor}
Samuel Karlin and James~L. McGregor.
\newblock The {H}ahn polynomials, formulas and an application.
\newblock {\em Scripta Math.}, 26:33--46, 1961.

\bibitem{Koekoek}
Roelof Koekoek, Peter~A. Lesky, and Ren{\'e}~F. Swarttouw.
\newblock {\em Hypergeometric orthogonal polynomials and their
  {$q$}-analogues}.
\newblock Springer Monographs in Mathematics. Springer-Verlag, Berlin, 2010.
\newblock With a foreword by Tom H. Koornwinder.

\bibitem{Letac}
G{\'e}rard Letac.
\newblock Probl\`emes classiques de probabilit\'e sur un couple de {G}el\cprime
  fand.
\newblock In {\em Analytical methods in probability theory ({O}berwolfach,
  1980)}, volume 861 of {\em Lecture Notes in Math.}, pages 93--120. Springer,
  Berlin-New York, 1981.

\bibitem{LPW}
David~A. Levin, Yuval Peres, and Elizabeth~L. Wilmer.
\newblock {\em Markov chains and mixing times}.
\newblock American Mathematical Society, Providence, RI, 2009.
\newblock With a chapter by James G. Propp and David B. Wilson.

\bibitem{Markov}
A.A. Markov.
\newblock Rasprostranenie zakona bol'shih chisel na velichiny, zavisyaschie
  drug ot druga.
\newblock {\em Izvestiya Fiziko-matematicheskogo obschestva pri Kazanskom
  universitete, 2-ya seriya}, 15(3):135–156, 1906.

\bibitem{Pulskamp}
Richard Pulskamp.
\newblock Notes to accompany disquisitiones analyticae de novo problemate
  coniecturali.
\newblock {\em
  http://cerebro.xu.edu/math/Sources/DanBernoulli/notes\%20to\%20de\%20novo\%20problemate.pdf}.

\bibitem{Defense}
Charles Rackoff and Daniel~R. Simon.
\newblock Cryptographic defense against traffic analysis.
\newblock In {\em Proceedings of the Twenty-fifth Annual ACM Symposium on
  Theory of Computing}, STOC '93, pages 672--681, New York, NY, USA, 1993. ACM.

\bibitem{Stanton}
Dennis Stanton.
\newblock Orthogonal polynomials and {C}hevalley groups.
\newblock In {\em Special functions: group theoretical aspects and
  applications}, Math. Appl., pages 87--128. Reidel, Dordrecht, 1984.

\bibitem{Todhunter}
I.~Todhunter.
\newblock {\em History of the Mathematical Theory of Probability from the Time
  of Pascal to that of Laplace}.
\newblock New critical idiom. Macmillan and Company, 1865.

\end{thebibliography}

\end{document}